\documentclass[12pt]{l4dc2020} 
\usepackage{algorithm2e}
\newtheorem{assumption}{Assumption}
\def\R{{\mathbb R}}
\def\E{{\mathbb E}}
\def\curlyA{{\mathcal A}}
\def\curlyB{{\mathcal B}}
\def\curlyR{{\mathcal R}}
\def\curlyS{{\mathcal S}}
\title[Convergence Analysis of Robust Deep Learning]{Robust Deep Learning as Optimal Control: \\ Insights and Convergence Guarantees}%Convergence Analysis of Robust Deep Learning with Optimal Control]{Convergence Analysis of Robust Deep Learning with Optimal Control
\usepackage{times}
\author{%
 \Name{Jacob H. Seidman} \Email{seidj@sas.upenn.edu}\\
 \Name{Mahyar Fazlyab} \Email{mahyarfa@seas.upenn.edu}\\
  \Name{Victor M. Preciado} \Email{preciado@seas.upenn.edu}
  \\
 \Name{George J. Pappas} \Email{pappasg@seas.upenn.edu}\\
}

\begin{document}

\maketitle

\begin{abstract}%
The fragility of deep neural networks to adversarially-chosen inputs has motivated the need to revisit deep learning algorithms. Including adversarial examples during training is a popular defense mechanism against adversarial attacks. This mechanism can be formulated as a min-max optimization problem, where the adversary seeks to maximize the loss function using an iterative first-order algorithm while the learner attempts to minimize it. However, finding adversarial examples in this way causes excessive computational overhead during training. By interpreting the min-max problem as an optimal control problem, it has recently been shown that one can exploit the compositional structure of neural networks in the optimization problem to  improve the training time significantly. In this paper, we provide the first convergence analysis of this adversarial training algorithm by combining techniques from robust optimal control and inexact oracle methods in optimization. Our analysis sheds light on how the hyperparameters of the algorithm affect the its stability and convergence. We support our insights with experiments on a robust classification problem.

\end{abstract}


\section{Introduction}

Deep neural networks have repeatedly demonstrated their capacity to achieve state of the art performance on benchmark machine learning problems \cite{lecun2015deep}. However, their performance can be significantly affected by small input perturbations that can drastically change the network's output \cite{szegedy2013intriguing}.  %Since this initial finding, there have been many methods proposed to craft adversarial attacks or hardening the classifiers against these attacks.
%methods of attack which can successfully fool a trained deep network into miss-classifying an input, among them \textcolor{red}{attack refs}.
%
In safety-critical applications the cost of such errors is prohibitive.  %In response, there have been many methods proposed to attack and defend neural network classifiers.
Therefore, an important line of work has emerged to train deep neural networks to be robust to adversarially-chosen perturbations.  

Among the most empirically successful methods is an optimization-based approach, where adversarial training is formulated as a min-max non-convex optimization problem \cite{madry2018towards}. To solve this problem, the adversary seeks to maximize the loss over sets of admissible perturbations, typically using an iterative method such as Projected Gradient Descent (PGD) \cite{madry2018towards}, the Fast Gradient Sign method (FGSM) \cite{goodfellowfgsm}, or other methods \cite{carlini2017towards}. The learner's goal is then to minimize the worst-case loss, as computed by the adversary, over the parameters of the neural network. In practice, however, the adversary can only approximate the worst-case loss. Additionally, each iteration of the adversary requires one backpropagation through the network. This results in a multiplicative factor increase in the number of backpropagations needed for training, which can significantly increase the total training time.

%given by norm balls of specified radius around each training data point. 
%which seeks to minimize the worst-case loss over sets of admissible perturbations,  given by norm balls of specified radius around each training data point. 
% Hey Jacob
%A major computational issue is finding the worst-case loss, since any iterative algorithm almost certainly fails at finding the global maximum of the loss function. Even if we can find a relatively good local maximum, the computational overhead incurred by solving the inner maximization problem for each training data point makes the overall training slow.
%The robust optimization problem is iteratively solved by finding adversarial perturbations between parameter updates of the model, whether by projected gradient descent (PGD), the fast gradient sign method (FGSM), or other known methods of attack, \cite{carlini2017towards}.  Using these perturbations on the training data set then gives state of the art robustness against future attacks \cite{madry2018towards}. However, this method of training comes at a computational cost. In the case of PGD or other first order methods, one must compute backpropagations through the network for as many iterations that are used to update the adversary.  This results in a multiplicative factor increase in the number of backpropogations needed for training which can significantly increase the total training time.

Nevertheless, it was shown in \cite{zhang2019you} that the computational cost for the adversary can be significantly reduced by exploiting the inherent compositional structure of deep neural networks.  In particular, by viewing a $T$-layer neural network as a discrete-time dynamical system with time horizon $T$, the min-max robust optimization problem can be seen as a finite-horizon robust optimal control problem.  In this interpretation, the adversary is finding the worst-case additive perturbation to the initial condition of the system (this is a special case of the $H_\infty$ control problem \cite{bacsar2008h}).  The learner then minimizes the worst-case cost function over the parameters of the network.  Deriving the necessary conditions for this robust control problem from the Pontryagin Maximum Principle (PMP) leads to algorithm proposed in \cite{zhang2019you}.  While the algorithm is empirically very successful, its convergence analysis has not yet been addressed. 

In this paper, we give the first convergence proof of this optimal control inspired robust training algorithm.  By viewing the adversary updates as being derived from the costate process of the deep network dynamics, we bound the error from the adversary's updates to its true gradients.  This allows us to appeal to results on first order methods with inexact oracles and prove a convergence result for this algorithm which explicitly shows the dependence on the algorithm parameters.  The argument we construct provides an outline for future results on the convergence of computationally efficient robust training algorithms. Our result further suggest that for a fixed number of backpropagations, increasing the number of adversary updates past a certain point can have a negative effect on performance.  This insight is supported by experiments on a robust classification problem.

% \section{Robust Training Problem Formulation}
% Consider a $T$-layer deep neural network $f_{\theta} \colon \mathbb{R}^{n_0} \to \mathbb{R}^{n_T}$ described by the recursive equations
\textbf{Preliminaries and Notation:} We denote by $\R^d$ the set of $d$-dimensional vectors with real valued components. The inner product is denoted $\langle \cdot, \cdot \rangle$ and the 2-norm is denoted $\|\cdot\|$.  We say a function $f$ is $L$-smooth if it has $L$-Lipschitz gradients.  For $\mu > 0$, a differentiable function $f$ is $\mu$-strongly concave if for all $x,y$, $f(x) \leq f(y) + \langle f(y),x-y\rangle - (\mu/2)\|x-y\|^2$.  If a function $f$ is $\mu$-strongly concave and $L$-smooth, then for all $x,y$, $\|x-x^\star\| \leq 1/\mu \|\nabla f(x)\| \leq (1/\mu) \sqrt{2 L(f(x^\star) - f(x))}$, where $x^\star = \mathrm{argmax}_x f(x)$ \cite{boyd2004convex}.  For a compact set $\mathcal X$ we define its diameter as $D(\mathcal X) := \max_{x,x' \in \mathcal X} \|x - x'\|$.

\section{Robust Training Problem Formulation}

Consider a $T$-layer deep neural network with hidden dimensions $n_1,\ldots,n_{T}$ described by $F(x,\theta)\colon$ $\mathbb{R}^{d_x} \times \mathbb{R}^{d_\theta} \to \mathbb{R}^{d_y}$, where $x$ is the input and $\theta$ are the trainable parameters.  We overload notation slightly and let the ``0-th'' layer have dimension $n_0 = d_x$ and the output layer have dimension $n_T = d_y$.
%
%Consider a $T$-layer deep neural network described by the recursion $x_{t+1}=f_t(x_t,\theta_t)$, where for layer $t$, $x_t \in \mathbb{R}^{n_t}$ is the input to the layer, $\theta_t \in \Theta_t \subseteq \mathbb{R}^{m_t}$ is the trainable parameter of the layer, and $f_t: \R^{n_t} \times \Theta_t \to \R^{n_{t+1}}$ is the state transition function. We denote by $F(x,\theta)=f_{T-1}(f_{T-2}(\cdots f_0(x, \theta_0) \cdots),\theta_{T-2})\theta_{T-1})$ the input-output map of the neural network, where $\theta$ is the concatenation of $\{\theta_0,\ldots, \theta_{T-1}\}$. 
Given a norm-based perturbation ball $\mathcal X$ and a training dataset $\curlyS = \{(x_{0,1},y_1)\ldots (x_{0,S},y_S) \}$ of size $S$, the robust training problem can be formulated as \citep{madry2018towards}
\begin{align} \label{original robust training}
\underset{\theta}{\text{minimize}} \sum_{i=1}^{S} \max_{\eta_i \in \mathcal X} \Phi(F(x_{0,i}+\eta_i,\theta),y_i),
\end{align}
where $\Phi: \R^{n_T} \times \R^{d_y} \to \R$ is a convex surrogate loss function penalizing the difference between the predicted and true labels. Throughout we will reserve $i$ as the data index.  

 %with $x_{0,i}\in \R^{n_0}$ and $y_i \in \R^{d_y}$ and a $T$-layer deep neural network with widths $\{n_0,\ldots,n_{T-1}\}$.  We denote by $f_t: \R^{n_t} \times \Theta_t \to \R^{n_{t+1}}$ the function corresponding to the $t$-th layer, where $\Theta_t$ is the set of parameters for the $t$-th layer.  Let $\ell_i: \R^{n_T} \times \R^{d_y} \to \R$ be a smooth function measuring the difference of a transformed version of the output of the network to the corresponding training label.  For example, in a classification problem $\ell_i$ would be the composition of a softmax function and the cross entropy loss for the label of the $i$-th example.  We denote by $\eta_i \in \R^{n_0}$ the perturbation to be applied to the $i$-th data point.  Our goal is to solve the following optimization problem
% \begin{align} \label{original robust training}
%     &\underset{\theta_1,\ldots,\theta_T}{\text{minimize}}\;\; \underset{\eta_1,\ldots \eta_S}{\text{maximize}} \quad  \sum_{i=1}^S F_i(x_{0,i} + \eta_i,\theta) + \sum_{i=1}^S \sum_{t=0}^{T-1} R_t(\theta_t) - \sum_{i=1}^N \frac{\beta}{2}\|\eta_i\|^2,
% \end{align}
% where $\theta$ is the concatenation of $\{\theta_0,\ldots, \theta_{T-1}\}$, $R_t$ is a pontential regularization term for the $t$-th layer, and 
% \begin{align}
%     F_i(x,\theta) &:= \ell_i\left(f_{T-1}(f_{T-2}(\cdots f_0(x, \theta_0) \cdots),\theta_{T-2})\theta_{T-1})\right).
% \end{align}

\section{An Optimal Control Inspired Algorithm}

%As originally proposed in \cite{weinan2017proposal}, 

Due to their compositional structure, feed-forward deep neural networks can be viewed as dynamical systems.  This approach has been taken recently in a number of papers which explore these dynamics and use the interpretation to suggest new training algorithms \cite{weinan2017proposal, li2017maximum, li2018optimal, weinan2019mean, zhang2019you}. Explicitly, we can describe a $T$-layer deep neural network $F(x,\theta)$ by the recursion $x_{t+1}=f_t(x_t,\theta_t)$, $t=0,\cdots,T-1$, 
where $x_t \in \mathbb{R}^{n_t}$ are the states (the output of the $t$-th layer), $f_t \colon \mathbb{R}^{n_{t}} \times \mathbb{R}^{m_t} \to \mathbb{R}^{n_{t+1}}$ is the state transition map, $\theta_t \in \mathbb{R}^{m_{t}}$ are the trainable control parameters, $\theta$ is the concatenation of $(\theta_i)_{0 \leq T-1}$\footnote{With this representation, the input-output map of the neural network is $F(x,\theta)=f_{T-1}(f_{T-2}(\cdots f_0(x, \theta_0) \cdots),\theta_{T-2})\theta_{T-1})$.}, and the initial conditions are given by the inputs to the network, $x_{0,i}$. Expressing the neural network as a dynamical system allows us to rewrite problem \eqref{original robust training} as the following optimal control problem:
\begin{alignat}{2} \label{robust training prob}
&\underset{\theta_1,\ldots,\theta_T}{\text{minimize}}\;\; \underset{\eta_1,\ldots \eta_S}{\text{maximize}} \; \quad && \sum_{i=1}^S \Phi(x_{T,i},y_i) + \sum_{i=1}^S \sum_{t=0}^{T-1} R_t(x_{t,i}, \theta_t) \\
&\text{subject to} && x_{t+1,i} = f_t(x_{t,i},\theta_t), \quad i=1,\ldots,S, \quad t = 1,\ldots, T-1 \nonumber \\
& 			&& x_{1,i} = f_0(x_{0,i} + \eta_i, \theta_0). \quad i=1,\ldots,S \nonumber
\end{alignat}
where $R_t$ is a potential regularizer on the states and controls for the $t$-th layer.  The two-player Pontryagin Maximum principle, proved in \cite{zhang2019you} gives necessary conditions for an optimal setting of the parameters $\theta^\star$, perturbations $\eta_1^\star, \ldots, \eta_S^\star$, and corresponding trajectories $\{x_{t,i}^\star\}$.  Define the Hamiltonians
\begin{align} \label{ham def}
    H_t(x,p,\theta) &:= p^\top f_t(x,\theta) - R_t(x,\theta), \quad t = 1,\ldots, T-1 \\
    H_0(x,p,\theta,\eta) &:= p^\top f_0(x+\eta,\theta) - R_0(x,\theta).
\end{align}
The two player maximum principle says in this case that if $\Phi$, $f_t$, and $R_t$ are twice continuously differentiable, with respect to $x$, uniformly bounded in $x$ and $t$ along with their partial derivatives, and the image sets $\{f_t(x,\theta)\;|\;\theta \in \R^{d_\theta}\}$ and $\{R_t(x,\theta)\;|\;\theta \in \R^{d_\theta}\}$ are convex for all $x$ and $t$, then there exists an optimal costate trajectory $p_{t,i}^\star$ such that the following dynamics are satisfied
\begin{align} \label{ham dynamics}
    x_{t+1,i}^\star &= \nabla_p H_t(x_{t,i}^\star,p_{t+1,i}^\star,\theta_t^\star), \quad x_{1,i}^\star = \nabla_p H_0(x_{0,i},p_{1,i}^\star,\theta_0^\star, \eta_i^\star) \\
    p_{t,i}^\star &= \nabla_x H_t(x_{t,i}^\star,p_{t+1,i}^\star,\theta_t^\star), \quad p_{T,i}^\star = -\nabla_x \Phi(x_{T,i}^\star,y_i), \label{eq: costate dynamics}
\end{align}
and the following Hamiltonian condition for all $\theta_t \in \R^{d_{\theta_t}}$ and $\eta_i \in \mathcal X$
\begin{align} \label{eq: ham max}
    \sum_{i \in \curlyS}H_t(x_{t,i}^\star, p_{t+1,i}^\star, \theta_t) &\leq \sum_{i \in \curlyS} H_t(x_{t,i}^\star, p_{t+1,i}^\star, \theta_t^\star), \quad t = 1,\ldots, T-1 \\
    \sum_{i \in \curlyS} H_0(x_{t,i}^\star, p_{t+1,i}^\star, \theta_t, \eta_i^\star) &\leq \sum_{i \in \curlyS} H_0(x_{t,i}^\star, p_{t+1,i}^\star, \theta_t^\star, \eta_i^\star)  \leq  \sum_{i \in \curlyS}H_0(x_{t,i}^\star, p_{t+1,i}^\star, \theta_t^\star, \eta_i).
\end{align}

These necessary optimality conditions can be used to design an iterative algorithm of the following form. For each data point $i \in \{1,\cdots,S\}$,
\begin{enumerate}
\item Compute the state and costate trajectories $\{x_{i,t}\}$ and $\{p_{i,t}\}$ from \eqref{eq: p dynamics}, keeping $\theta_t$ and $\eta_i$ fixed:
\begin{align} \label{eq: p dynamics}
x_{t+1,i}^\eta &= \nabla_p H_t(x_{t,i}^\eta,p_{t+1,i}^\eta,\theta_t), \quad x_{1,i}^\eta = \nabla_p H_0(x_{0,i},p_{1,i}^\eta,\theta_0, \eta) \\
    p_{t,i}^\eta &= \nabla_x H_t(x_{t,i}^\eta,p_{t+1,i}^\eta,\theta_t), \quad p_{T,i}^\eta = -\nabla_x \Phi(x_{T,i}^\eta,y_i). \label{eq: p dynamics 2}
\end{align} 
\item Minimize the Hamiltonian $H_0(x_{t,i}, p_{t+1,i}, \theta_t, \eta_i)$ with respect to $\eta_i$.
\item Maximize the sum of Hamiltonians $\sum_{i \in \curlyS} H_t(x_{t,i}, p_{t+1,i},\theta_t)$ with respect to $\theta_t$ for all $t$.
\end{enumerate}

%We will occasionally use the notation $\{x_{t,i}^\eta\}_{t=1}^T$ and $\{p_{t,i}^\eta\}_{t=1}^T$ to describe the trajectories that come from the Hamiltonian dynamics with initial condition $x_{0,i}$ and adversarial perturbation $\eta$,

 As was noticed as early as \cite{lecun1988theoretical}, it can be seen from the chain rule that the backward costate dynamics in \eqref{eq: p dynamics 2} are equivalent to backpropagation through the network. With this interpretation, the gradient of the total loss for the $i$-th data point with respect to the adversary $\eta_i$ can be written as $\nabla_\eta f_0(x_{0,i} + \eta_i,\theta_0)^\top p_{1,i}^{\eta_i}$. For a fixed value of $\theta_0$, performing gradient descent on $H_0$ to find a worst-case adversarial perturbation can be expressed as the following updates, where $\alpha > 0$ is a step size and we have for the moment dropped the dependence on the data index $i$.
\begin{align} \label{adversary gd}
    \eta^{\ell + 1} = \eta^\ell - \alpha \nabla_\eta f_0(x_0 + \eta^\ell,\theta_0)^\top p_1^{\eta^\ell}.
\end{align}
An important observation made in \cite{zhang2019you} is that the adversary is only present in the first layer Hamiltonian condition and this function can be minimized by computing gradients only with respect to the first layer of the network.  More explicitly, instead of using $p_1^{\eta^\ell}$, as in the updates in \eqref{adversary gd}, we could instead use $p_1^{\eta^0}$ and the updates
\begin{align} \label{yopo eta update}
    \eta^{\ell+1} = \eta^{\ell} - \alpha\nabla_\eta f_0(x_{0} + \eta^{\ell},\theta_0)^\top p_{1}^{\eta^0}.
\end{align}

This removes the need to do a full backpropagation to recompute the costate $p_1^{\eta^\ell}$ for every update of $\eta^\ell$, at the cost of now being an approximate gradient.  In other words, we work with ``frozen gradients'' of the later layers.  This inspires the ``YOPO-$m$-$n$'' (You Only Propogate Once) algorithm in \cite{zhang2019you}, where the adversary is updated with $m$ full backpropagations, after each of which $n$ updates of the form \eqref{yopo eta update} are performed.  A modified version of this method is written in pseudocode with the Hamiltonian framework in mind in Algorithm \ref{alg: YOPO}.  While in \cite{zhang2019you}, Algorithm \ref{alg: YOPO} was shown to have very promising empirical results, in this paper we provide a rigorous convergence analysis of its behavior.

\begin{algorithm2e}[t] \label{alg: YOPO}
  \SetAlgoLined
  Initialize $\theta^0$ randomly\;
  \For{$k=1,2,\ldots$}{
    	Randomly select mini-batch $\mathcal B$\;
    	Randomly initialize $\eta_i^{0,0} \in \mathcal X$, $i \in \{1,\ldots, B\}$\;
    	\For{$j = 0,\ldots, m-1$}{
    		$x_{1,i} \leftarrow \nabla_p H_0(x_{0,i},p_{1,i},\theta_0,\eta_i^{j,0})$, \quad  $i \in \{1,\ldots, B\}$\;
		\For{$t = 1, \ldots, T-1$}{
			$x_{t+1,i} \leftarrow \nabla_p H_t(x_{t,i},p_{t+1,i},\theta_t)$ \;
		}
		$p_{i,T} \leftarrow -\frac{1}{B} \nabla \Phi(x_{T,i},y_i)$,\quad  $i \in \{1,\ldots, B\}$\;
		\For{$t=T-1,\ldots,1$}{
			$p_{t,i} \leftarrow \nabla_x H_t(x_{t,i},p_{t+1,i},\theta_t)$, \quad $i \in \{1,\ldots, B\}$\;
		}
		\For{$\ell=0,\ldots, n-1$}{
			$\eta_{i}^{j,\ell+1} \leftarrow \Pi_\mathcal{X}\left[\eta_{i}^{j,\ell} - \alpha\nabla_\eta H_0(x_{0,i},p_{1,i},\theta_0,\eta_i^{j,\ell})\right]$, \quad $i \in \{1,\ldots, B\}$\;
		}
    	}
	$\theta^{k+1} \leftarrow \theta^k - \gamma_t \frac{1}{B} \sum_{i=1}^B \nabla_\theta \Phi(F(x_{i,0} + \eta_i^{m,n},\theta^k),y_i)$\;
    }
\caption{You Only Propagate Once (YOPO-m-n) Robust Traning Algorithm}
\end{algorithm2e}

\section{Convergence Analysis of Adversarial Training}

To prove convergence we interpret Algorithm \ref{alg: YOPO} as consisting of two nested gradient methods with inexact gradient oracles.  The inner method finds an adversarial perturbation by performing gradient descent on the Hamiltonian $H_0$ with frozen gradients at layers 2 through $T$, or equivalently, with a frozen costate $p_1$.  Not updating this costate at every iteration is what creates the oracle error for the adversary's problem.  By bounding the difference of the frozen costate to the actual costate, we are able to bound the oracle error and appeal to known inexact oracle convergence results for the adversary's problem.

The outer method then makes a parameter update to the network based on the perturbation found by the inner method.  If the inner method found the true worst case perturbation, this would result in an exact gradient update.  However, since in general the adversary's inner method will not converge to the true optimal point in finitely many iterations, and is an inexact method itself, the update for the network parameters can also be seen as coming from an inexact gradient oracle. Using the convergence result for the adversary then lets us bound the oracle error for the outer method, and we can then complete the proof with known techniques for convergence of gradient descent on non-convex functions with an inexact oracle.  All proofs are deferred to the appendix.   

We first set up some notation. For a given data point $i$, let $\mathcal A_i(\eta,\theta) := 
\Phi(F(x_{0,i} + \eta, \theta),y_i)$. Let $\eta^\star_i(\theta):= \mathrm{argmax}_{\eta} \curlyA_i(\eta,\theta)$.  We define the robust loss function $\curlyR (\theta) := (1/S)\sum_{i=1}^S \curlyA_i(\eta_i^\star(\theta),\theta)$.  Let $\mathcal B$ indicate a sampled mini-batch of the data of size $B$.  Let $g_\curlyB(\theta) = (1/B)\sum_{i\in \curlyB} \nabla_\theta \curlyA(\eta_i^{\star}(\theta),\theta)$ denote the corresponding stochastic gradient of the robust loss.  Note that $\E[g_\curlyB(\theta)] = \nabla_\theta \curlyR(\theta)$ where the expectation is taken over the randomness of the mini-batch sampling.  

%The authors of \cite{wang2019convergence} introduced the following First-Order Stationarity Condition as a quantitative measure of the solution found for the inner minimization problem,
%\begin{align} \label{FOSC}
%c_x(\eta') = \underset{\eta \in \mathcal X}{\text{max}}\;  \left\langle \eta - \eta', \nabla_\eta F(x + \eta',\theta)\right\rangle.
%\end{align}
%This condition is inspired from the Frank-Wolfe optimality condition.  Note that $c_x(\eta) \geq 0$ for all $\eta \in \mathcal X$. It was shown in \cite{wang2019convergence} that smaller values of $c_x(\eta)$ correspond with stronger adversarial examples.
%
We now present the assumptions that will be in place for the theoretical results of this paper.

\begin{assumption} \label{asu: lip}
There exists a constant $K > 0$ such that for all $t \in 1,\ldots, T$, the functions $f_t$, $\Phi$, $\nabla_x f_t$, and $\nabla_x R_t$ are $K$-Lipschitz in $x$, uniformly in $\theta$.
For all $i = 1, \ldots, S$, the functions $\nabla_\theta \curlyA_i$ and $\nabla_\eta \curlyA_i$ satisfy the following Lipschitz conditions,
\begin{align}
&\|\nabla_\theta \curlyA_i(\eta,\theta^1) - \nabla_\theta \curlyA_i(\eta,\theta^2)\| \leq L_{\theta \theta} \| \theta^1 - \theta^2\|   \\
&\|\nabla_\theta \curlyA_i(\eta^1,\theta) - \nabla_\theta \curlyA_i(\eta^2,\theta)\| \leq L_{\theta \eta} \| \eta^1 - \eta^2\| \\
&\|\nabla_\eta \curlyA_i(\eta,\theta^1) - \nabla_\eta \curlyA_i(\eta,\theta^2)\| \leq L_{\eta \theta} \| \theta^1 - \theta^2\|   \\
&\|\nabla_\eta \curlyA_i(\eta^1,\theta) - \nabla_\eta \curlyA_i(\eta^2,\theta)\| \leq L_{\eta \eta} \| \eta^1 - \eta^2\|
\end{align}
\end{assumption}

Such Lipschitz assumptions are standard in the optimization literature.  Note that the assumption of the existence of the gradients in $x$ and $\eta$ of functions of the network restricts the potential activation functions of the network to not include the ReLU function, though it does allow for sigmoid, tanh, and ELU activations. Leveraging these smoothness assumptions will be essential in proving the rate of convergence that follows.  

\begin{assumption} \label{asu: sc}
 $\curlyA_i(\eta,\theta)$ is locally $\mu$-strongly concave for $\eta \in \mathcal X$, that is for any $\theta$ and $\eta^1, \eta^2 \in \mathcal X$,
 \begin{align}
     \curlyA_i(\eta^1,\theta) \leq \curlyA(\eta^2,\theta) + \langle \curlyA_i(\eta^2,\theta), \eta^1 - \eta^2\rangle - \frac{\mu}{2}\|\eta^1 - \eta^2\|^2.
 \end{align}
\end{assumption}

This assumption was made in previous results on convergence of robust training \cite{wang2019convergence} and is justified through the reformulation of robust training as distributionally robust optimization \cite{sinha2018certifiable, lee2018minimax}.  Perturbing each data point in the $\ell^p$ norm by $\epsilon$ results in perturbing the empirical distribution in the $p$-Wasserstein distance by at most $\epsilon$.

\begin{assumption} \label{asu: variance of sg}
The stochastic gradients satisfy $\E\left[\|g_\curlyB(\theta) - \nabla \curlyR(\theta)\|^2\right] \leq \sigma^2$, with $\sigma \geq 0$.
\end{assumption}

This assumption is standard in convergence results for optimization algorithms with noisy gradients.
It was shown in \cite{sinha2018certifiable} that under these assumptions the robust loss function has Lipschitz gradients and the following relation holds.  This will allow us to use techniques for convergence of gradient descent on non-convex functions.
\begin{proposition}[\cite{sinha2018certifiable}]\label{robust loss lip}
Under Assumptions \ref{asu: lip} and \ref{asu: sc}, the robust loss function $\curlyR (\theta)$ is $L$-smooth, where $L = L_{\theta \theta} + (L_{\theta \eta}L_{\eta \theta}/\mu)$, and the following inequality holds for all $\theta_1,\theta_2$,
\begin{align} \label{robust loss L smoothness}
\curlyR(\theta_1) \leq \curlyR(\theta_2) + \langle \nabla \curlyR(\theta_2),\theta_1 - \theta_2 \rangle + \frac{L}{2} \|\theta_1 - \theta_2\|^2.
\end{align}
\end{proposition}

We next derive the following three results used to prove our main theorem.  The first result bounds the difference between the costate used for the adversary's update, as in \eqref{yopo eta update}, and the costate that would result in a true gradient update, as in \eqref{adversary gd}.  The proof shows that the costates are Lipschitz as a function of the initial condition of the system, and then uses a bound on successive values of the perturbation $\eta$ from the adversary's updates.

\begin{lemma} \label{lem: del p bound}
There exists a constant $C'$ dependent on $T$ and $K$ such that for all $\ell \in \{0,\ldots, n-1\}$, $j \in \{0,\ldots, m\}$, and $i \in \{1,\ldots, S\}$
\begin{align}\label{del p1 bound}
\|p_{1,i}^{\eta_i^{j,0}} - p_{1,i}^{\eta_i^{j,\ell}}\| \leq C' \alpha (n-1).
\end{align}
\end{lemma}

Hence, we are able to bound the error incurred from the frozen costates of the adversary's updates to the true gradients.  In doing so, we can appeal to convergence results for inexact oracles and prove the next theorem on convergence of the adversary to a worst case perturbation. 

\begin{theorem} \label{thm nesty}
Under Assumptions \ref{asu: lip}, \ref{asu: sc}, and \ref{asu: variance of sg}, for a fixed value of $\theta$ and fixed data point $i$, let
\begin{align}
\hat \eta_i = \underset{\substack{j=1,\ldots,m \\ \ell = 1,\ldots, n}}{\mathrm{argmin}} \;\;\| \nabla_\eta \curlyA_i(\eta_i^{j,\ell},\theta)\|.
\end{align}
Then, if we define $C = K C'$ and set $\alpha < 1/L_{\eta \eta}$, then
\begin{align} \label{eq: adversary convergence}
\| \nabla_\eta \curlyA_i(\hat \eta_i,\theta)\|^2\leq D(\mathcal X)L_{\eta \eta}^2  \left(1 - \frac{\mu}{L_{\eta \eta}}\right)^{mn+1} + \frac{2 C^2}{L_{\eta \eta}} (n-1)^2\left(\frac{2}{\mu} + \frac{1}{2L_{\eta \eta}}\right),
\end{align}
\end{theorem}

The last intermediate result we use to prove our theorem relates how the suboptimality of the chosen adversarial perturbation bounds the error for the computed gradients of the robust loss.  To prove our main theorem we will apply the bound from the previous result to the following lemma.

\begin{lemma} \label{lem: approx gradients}
Under Assumptions \ref{asu: lip} and \ref{asu: sc}, if $\eta_i$ are such that $(1/B) \sum_{i\in \curlyB}\| \nabla_\eta \curlyA_i(\eta_i,\theta)\|^2 \leq \delta$ then
\begin{align} \label{eq: approx grad}
\left\| \frac{1}{B} \sum_{i\in \curlyB} \nabla_\theta \curlyA_\curlyB(\eta_i,\theta) - g_\curlyB(\theta) \right\| \leq \frac{L_{\theta \eta} \delta}{\mu}.
\end{align}
\end{lemma}

Combining the previous three results allows us to prove the main theorem, stated below.

\begin{theorem}[Convergence Analysis of Adversarial Training] \label{our theorem}
Under Assumptions \ref{asu: lip}, \ref{asu: sc}, and \ref{asu: variance of sg}, if the the step sizes $\gamma_t$ satisfy $\gamma_t = \gamma = \min\{1/L, \sqrt{\Delta/(L \sigma^2 N)}\}$ where $\Delta = \curlyR(\theta^0) - \inf_\theta \curlyR(\theta)$, $\alpha < 1/L_{\eta \eta}$, and the parameters are updated with the perturbations $\eta_i = \underset{\j\in\{1,\ldots,m\},\ell \in \{ 1,\ldots, n\}}{\mathrm{argmin}} \;\;\| \nabla_\eta \curlyA_i(\eta_i^{j,\ell},\theta)\|,$ then there exists a constant $C$ depending on $T$ and $K$ such that the iterates of YOPO-$m$-$n$ satisfy
\begin{align}\label{YOPO convergence}
\frac{1}{N} \sum_{k=1}^N \E\left[\|\nabla \curlyR(\theta^k)\|^2\right] \leq 4 \sigma \sqrt{\frac{L \Delta}{N}} + \frac{5 L_{\theta \eta}^2}{\mu}\left(D(\mathcal X)L_{\eta \eta}^2  \left(1 - \frac{\mu}{L_{\eta \eta}}\right)^{mn+1}\!\!\!\! + \frac{2C^2}{L_{\eta \eta}} \left(\frac{2}{\mu} + \frac{1}{2L_{\eta \eta}}\right)(n-1)^2\right).
\end{align}
\end{theorem}

The first term on the right side is typical for convergence of first order optimization algorithms on smooth non-convex functions, and is the same as the first term that appears in the convergence result of \cite{wang2019convergence}.  The second term represents the errors from both inexact oracles that accumulate over the algorithm.  The expression
\begin{align}
    \mathcal E(m,n) := D(\mathcal X)L_{\eta \eta}^2  \left(1 - \frac{\mu}{L_{\eta \eta}}\right)^{mn+1} + \frac{2C^2}{L_{\eta \eta}} \left(\frac{2}{\mu} + \frac{1}{2L_{\eta \eta}}\right)(n-1)^2,
\end{align}
shows how the solution to the adversary's problem contributes to the gradient oracle error for the parameter updates.  The first term represents the approximate nature of the adversary's solution, as it has finitely many iterations to maximize the loss function.  The second term shows the accumulation of gradient oracle errors for the adversary due to freezing the costate in between backpropagations.

Using this bound we can investigate the dependence of the algorithm on the number of backpropagations for the adversary, $m$, and the number of gradient steps taken with each frozen gradient, $n$.  We see that $\mathcal E$ monotonically decreases in $m$, implying that a practitioner should set $m$ to be as large as can be tolerated according to their computational budget. Thus, we will focus on the dependence of $\mathcal E$ on the number of adversary updates per backpropagation, $n$.

First we note that $\mathcal E$ is convex in $n$, as can be confirmed by computing its second partial derivative and observing $\partial^2 \mathcal E/\partial n^2 \geq 0$.  As $\partial \mathcal E/\partial n$ is monotonically increasing, we should only increase $n$ up to just before the point where $\partial \mathcal E/\partial n$ becomes positive.  This happens when
\begin{align}
     -\log\left(1 - \frac{\mu}{L_{\eta \eta}}\right)D(\mathcal X)L_{\eta \eta}^2 m\left(1 - \frac{\mu}{L_{\eta \eta}}\right)^{mn+1} \leq \frac{4C^2}{L_{\eta \eta}}\left(\frac{2}{\mu} + \frac{1}{2L_{\eta \eta}}\right) (n-1),
\end{align}
that is, when the exponentially decaying factor in $n$ on the left side overtakes the linearly growing factor in $n$ on the right side.  Therefore, our bound suggests that when $n$ is too large we will obtain lower robust accuracy, even though the adversary is given more updates to find a worst case perturbation.  We demonstrate this phenomenon with a robust classification experiment on the MNIST dataset, as shown in Figure \ref{fig}. 

\begin{figure}[!tbp] \label{fig}
  \centering
  \begin{subfigure}{\includegraphics[width=0.49\textwidth]{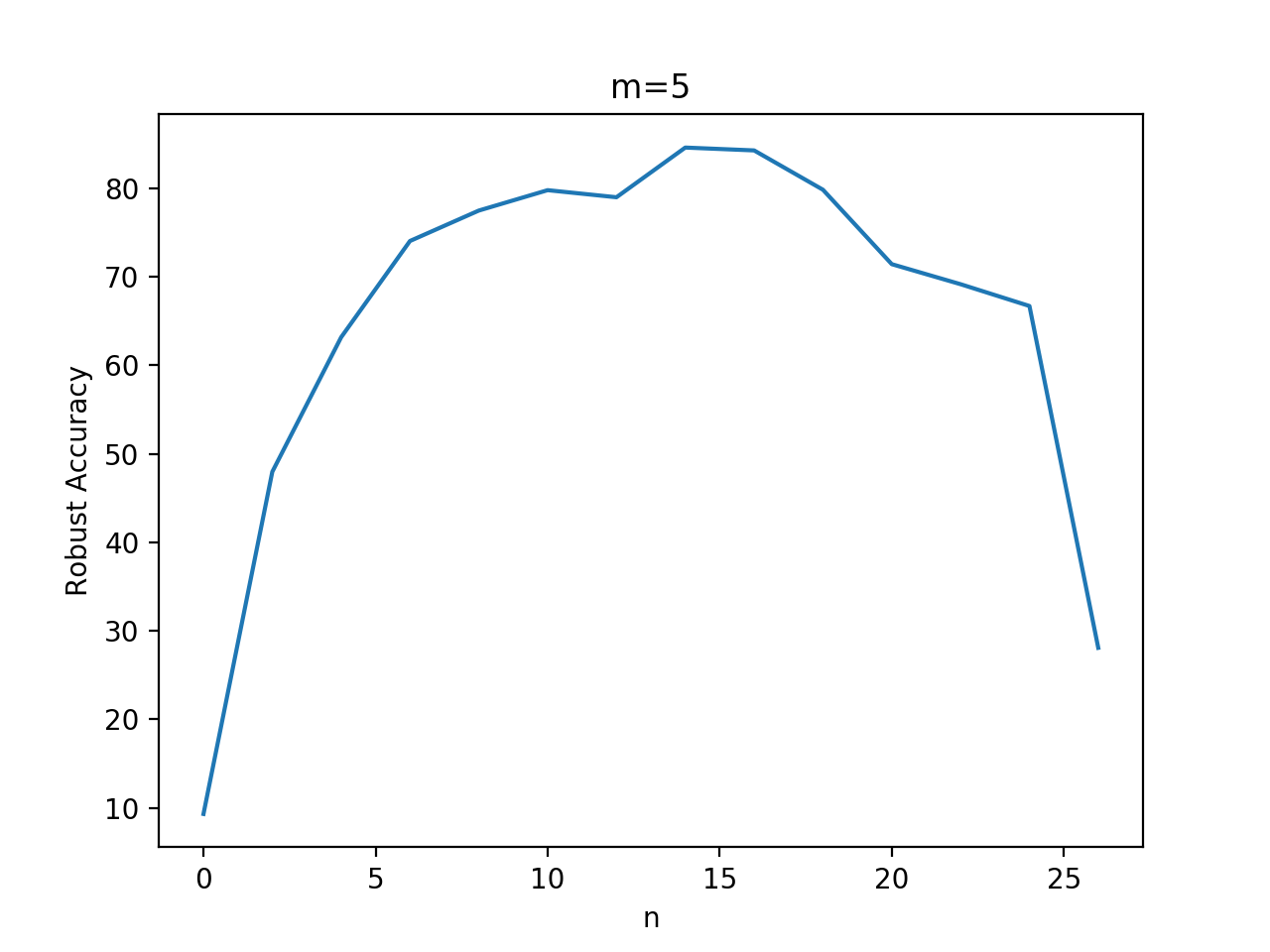}}
  \end{subfigure}
  \hfill
  \begin{subfigure}{\includegraphics[width=0.49\textwidth]{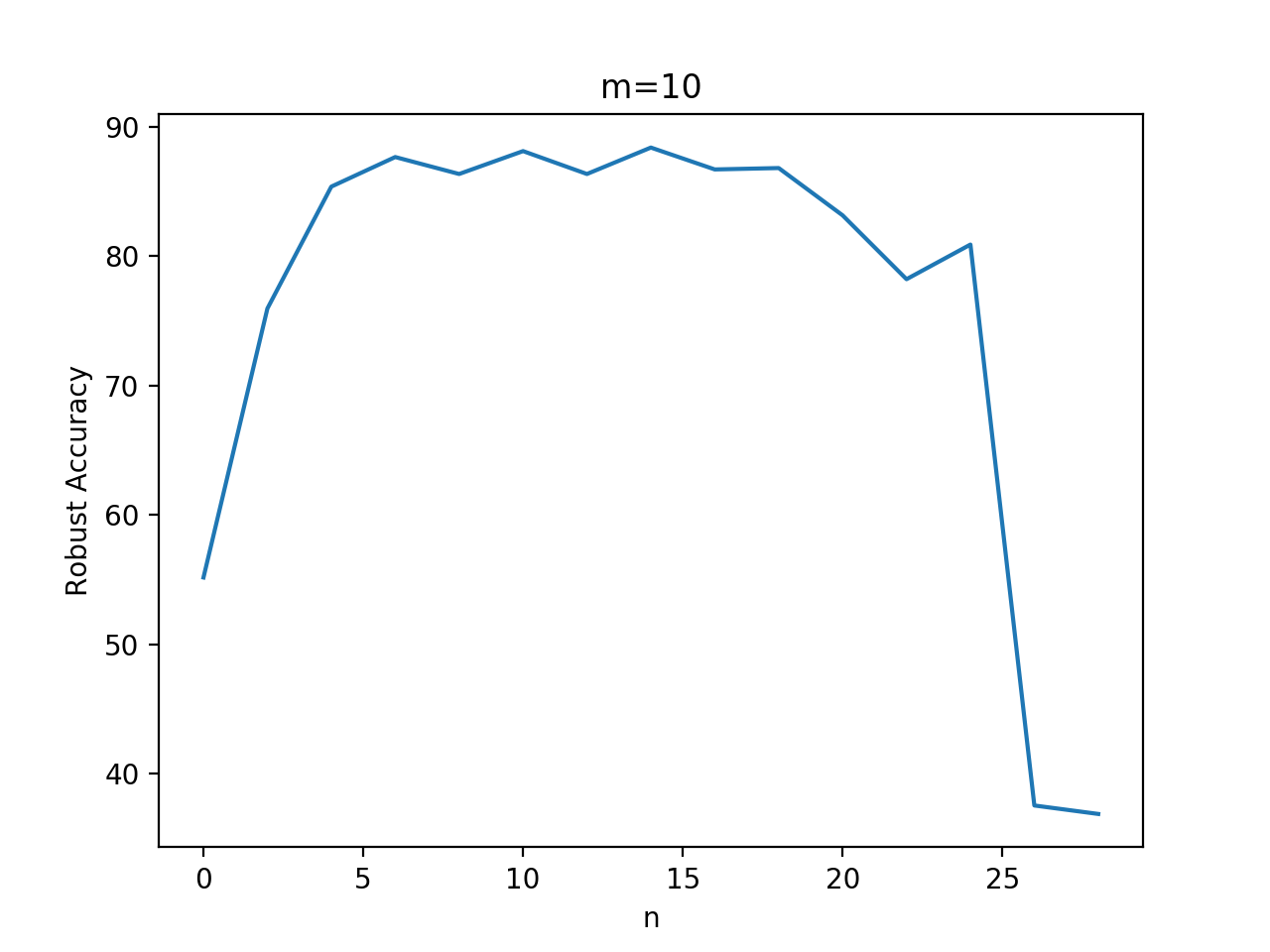}}
  \end{subfigure}
  \caption{Robust Accuracy after training with YOPO-m-n after 10 epochs.  Figure on the left shows how accuracy changes with $m=5$ fixed and varying $n$, figure on the right is the same for $m=10$ and varying $n$.  In both figures we see that performance degrades quickly in $n$ after a certain point, as predicted by Theorem \ref{our theorem}.}
\end{figure}

This observation is reminiscent of results in the literature on the Method of Successive Approximations (MSA) for finding controls and trajectories which satisfy the PMP.  These methods alternate between computing state and costate trajectories and maximize the Hamiltonian to update the control.  We can interpret the adversary's updates as a MSA variant for the adversary's Hamiltonian minimization condition.  It has been shown that if the new controls result in trajectories that deviate too far from the trajectories used in the Hamiltonian (in our case resulting in a larger oracle error) these methods will not converge \cite{chernousko1982method}. This is consistent with the interpretation of our result.

\section{Conclusion}

%Feed-forward neural networks can be viewed as dynamical systems. This interpretations allows one to exploit the compositional structure of deep neural networks to improve

We give the first convergence analysis for a recently proposed robust training algorithm for deep neural networks.  By using methods from optimal control theory and results from inexact oracle methods in optimization, we shed light on the behavior of the algorithm as a function of its hyperparamters.  It is likely that the interpretation of PMP-based algorithms as inexact oracle methods can be used to prove convergence for other learning algorithms inspired by optimal control, such as the MSA variants proposed in \cite{li2018optimal}; this is left for future work.  Another avenue to explore is the behavior of approximate adversary updates in the overparameterized regime, as inspired by recent convergence results of overparameterized adversarial training with vanilla PGD \cite{gao2019convergence}.

\section{Appendix}

\begin{lemma} \label{lem: approx gradients}
Under Assumptions \ref{asu: lip} and \ref{asu: sc}, if $\eta_i$ are chosen such that $(1/B) \sum_{i\in \curlyB}\| \nabla_\eta \curlyA_i(\eta_i,\theta)\|^2 \leq \delta$ then
\begin{align} \label{eq: approx grad}
\left\| \frac{1}{B} \sum_{i\in \curlyB} \nabla_\theta \curlyA_\curlyB(\eta_i,\theta) - g_\curlyB(\theta) \right\| \leq \frac{L_{\theta \eta} \delta}{\mu}.
\end{align}
\end{lemma}

\begin{proof}
Using the definitions of $\curlyA_\curlyB$ and $g_\curlyB$,
\begin{align}
\left\| \frac{1}{B} \sum_{i\in \curlyB} \nabla_\theta \curlyA_\curlyB(\eta_i,\theta) - g_\curlyB(\theta) \right\| &\leq \frac{1}{B} \sum_{i \in \curlyB} \| \nabla_\theta \curlyA_i(\eta_i,\theta) - \nabla_\theta \curlyA_i(\eta_i^\star(\theta),\theta)\| \\
& \leq \frac{L_{\theta\eta}}{B} \sum_{i \in \curlyB} \| \eta_i - \eta_i^\star(\theta)\| \\
&\leq \frac{L_{\theta\eta}}{B \mu} \sum_{i \in \curlyB} \| \nabla \curlyA_i(\eta_i,\theta)\| \\
&\leq \frac{L_{\theta\eta} \delta}{\mu},
\end{align}
where the second to last inequality follows from the $\mu$-strong concavity of $\curlyA_i$ with respect to $\eta$.
\end{proof}

%The next result we need to has been applied in \cite{wang2019convergence}, and \cite{sinha2018certifiable}, the proof of which can be traced back to \cite{ghadimi2013stochastic}.  It states that under smoothness assumptions on the loss function, gradient descent with approximate gradients converges to a neighborhood of a stationary point at a rate $O(1/\sqrt{T})$.
%
%\begin{theorem} \label{generalized wang theorem}
%Under Assumptions \ref{asu: lip} and \ref{asu: sc}, consider the sequence of updates of the parameter $\theta$ given by
%\begin{align} \label{theta updates}
%\theta^{k+1} = \theta^k - \gamma_t \hat g(\theta^k),
%\end{align}
%where $\hat g(\theta^k)$ is an approximate stochastic gradient in the sense that for all $\theta$
%\begin{align} \label{approx g condition}
%\| \hat g(\theta) - g(\theta)\| \leq \epsilon.
%\end{align}
%If the step size $\gamma_t$ is set such that $\gamma_t = \gamma = \emph{min}\{1/K, \sqrt{\Delta/(K \sigma^2 N)}\}$ where $\Delta = L(\theta^0) - \inf_\theta L(\theta)$, then the iterates given by \eqref{theta updates} satisfy
%\begin{align}\label{theta convergence}
%\frac{1}{N} \sum_{k=1}^N \E\left[\|\nabla L(\theta)\|^2\right] \leq 4 \sigma \sqrt{\frac{K \Delta}{N}} + 5\epsilon^2.
%\end{align}
%\end{theorem}
%
%

Next, we present a lemma bounding the difference of the co-states of the first layer in the adversary's inner loop.  This will allow us to prove a convergence result for the adversary.  We fix the outer loop index at arbitrary $j$, the data point $i$ and for ease of notation drop the dependence of state variables on the index $i$.  We define $x_t^{\eta}$ and $p_t^{\eta}$ as the state and co-state trajectories generated from the initial condition $x_0 + \eta$. Define $x_t^\ell := x_t^{\eta^{j,\ell}}$,  $p_t^{\ell} := p_t^{\eta^{j,\ell}}$, and $\eta^{\ell} := \eta^{j,\ell}$.  We will additionally define $\delta p_t^\ell := p_t^{0} - p_t^{\ell}$ and $\delta x_t^\ell := x_t^0 - x_t^{\ell}$.

%%lemma boundedness of del p
\begin{lemma} \label{lem: del p bound}
There exists a constant $C'$ dependent on $T$ and $K$ such that for all $\ell \in \{0,\ldots, n-1\}$, $j \in \{0,\ldots, m\}$, and $i \in \{1,\ldots, S\}$
\begin{align}\label{del p1 bound}
\|p_{1,i}^{\eta_i^{j,0}} - p_{1,i}^{\eta_i^{j,\ell}}\| \leq C' \alpha (n-1).
\end{align}
\end{lemma}
%%proof
\begin{proof}
We will suppress the notational dependence on $\theta$ for all functions, as $\theta$ is fixed during the updates for the adversary $\eta$. We first prove bounds on $\|p_t^\ell\|$ and $\|\delta x_t^\ell\|$.  From the costate dynamics we have $\|p_T^\ell\| \leq  \|- \nabla \Phi(x_T^\ell,y)\| \leq K$, and
\begin{align}
\|p_t^{\ell}\| = \|\nabla_xH_t(x_t^\ell,p_{t+1}^\ell,\theta_t)\| \leq \|p_{t+1}^\ell\| \|\nabla_x f_t(x_t^\ell,\theta_t)\|  + \|\nabla_x R_t(x_t^\ell)\| \leq K \|p_{t+1}^\ell\| + K, 
\end{align}
so by induction
\begin{align} \label{eq: p bound}
    \| p_t^{\ell} \|\leq K + K^2 + \ldots + K^{T-t+1} \leq K^{T-t+1}(T-t+1).
\end{align}
Next, from Assumption \ref{asu: lip} we have that $\|\delta x_1^\ell\| = \| f_0(x_0 + \eta^0) - f_0(x_0 + \eta^\ell)\| \leq K \|\eta^0 - \eta^\ell\|$, so by induction we have
\begin{align} \label{eq: del x bound}
\|\delta x_t^\ell\| \leq K^t \|\eta^0 - \eta^{\ell}\|.
\end{align}
To bound $\|p_1^{0} - p_1^{\ell}\|$ we first note that $\|\delta p_T^{\ell}\| = \| \nabla \Phi(x_T^{\ell}) - \nabla \Phi(x_T^0)\| \leq K \|\delta x_T^\ell\| $.  We write
\begin{align}
\|\delta p_t^\ell\| &= \|\nabla_x H_t(x_t^0,p_{t+1}^0) - \nabla_x H_t(x_t^{\ell}, p_{t+1}^{\ell})\| \nonumber\\
&= \|\nabla_x H_t(x_t^0,p_{t+1}^0) - \nabla_x H_t(x_t^{\ell},p_{t+1}^0) + \nabla_x H_t(x_t^{\ell},p_{t+1}^0) -  \nabla_x H_t(x_t^{\ell}, p_{t+1}^{\ell})\| \nonumber\\
&= \|\langle p_{t+1}^0, \nabla_xf_t(x_t^0) - \nabla_x f_t(x_t^{\ell})\rangle + \langle p_{t+1}^0 - p_{t+1}^{\ell},\nabla_x f_t(x_t^{\ell})\rangle + \nabla_x R_t(x_t^{\ell}) - \nabla_x R_t(x_t^{0})\| \nonumber\\
&\leq \|p_{t+1}^0\| \|\nabla_x f_t(x_t^0) \!-\! \nabla_x f_t(x_t^{\ell})\| +\| p_{t+1}^0 - p_{t+1}^{\ell}\| \|\nabla_x f_t(x_t^{\ell})\| + \|\nabla_x R_t(x_t^{\ell}) \!-\! \nabla_x R_1(x_t^{0})\| \nonumber\\
&\leq (K^{T-t+1}(T-t)+K) \| \delta x_t^\ell\| + K \|\delta p_{t+1}^\ell\|,
\end{align}
where in the last line we have used \eqref{eq: p bound} and the Lipschitz assumptions on $f$, $\nabla_x f$, and $L$.  An application of the discrete Gronwall inequality then gives
\begin{align}
\| \delta p_t^\ell \| \leq K^{T-1} \left(K \| \delta x_T^\ell\| + \sum_{t=1}^{T-1} (K^{T-t+1}(T-t)+K)\| \delta x_t^\ell \|\right).
\end{align}
Applying \eqref{eq: del x bound} gives
\begin{align} \label{del p bound by eta diff}
\| \delta p_1^\ell \| \leq( K^T + T(T-1)K^{2T-2}+TK^{2T} ) \| \eta^0 - \eta^{\ell}\|.
\end{align}
Lastly, we may bound the right side of \eqref{del p bound by eta diff} by recalling that $\eta^{\ell}$ is obtained from $\eta^{0}$ after $n$ updates of the form $\eta^{s+1} = \eta^{s} + \alpha \nabla_\eta f_0(x+ \eta^{s},\theta)^\top p_1^{0}$. An application of \eqref{eq: p bound} and the Lipschitz assumption of $f_0$ then give
\begin{align} \label{eq: eta diff bound}
\| \eta^0 - \eta^{\ell}\| \leq K^{T+1} \alpha (n-1),
\end{align}
Combining \eqref{eq: eta diff bound} with \eqref{del p bound by eta diff} and defining $C' = K^{T+1}(K^T + T(T-1)K^{2T-2}+TK^{2T})$ completes the proof.
\end{proof}

%%%%Nesty lemma
\begin{theorem} \label{thm nesty}
Under Assumptions \ref{asu: lip}, \ref{asu: sc}, and \ref{asu: variance of sg}, for a fixed value of $\theta$ and fixed data point $i$, let
\begin{align}
\hat \eta_i = \underset{\substack{j=1,\ldots,m \\ \ell = 1,\ldots, n}}{\mathrm{argmin}} \;\;\| \nabla_\eta \curlyA_i(\eta_i^{j,\ell},\theta)\|.
\end{align}
Then, if we define $C = K C'$,
\begin{align} \label{eq: adversary convergence}
\| \nabla_\eta \curlyA_i(\hat \eta_i,\theta)\|^2\leq D(\mathcal X)L_{\eta \eta}^2  \left(1 - \frac{\mu}{L_{\eta \eta}}\right)^{mn+1} + \frac{2C^2}{L_{\eta \eta}} \alpha^2 (n-1)^2\left(\frac{2}{\mu} + \frac{1}{2L_{\eta \eta}}\right),
\end{align}
\end{theorem}
%%Nesty proof
\begin{proof}
We drop the dependence of all functions on $\theta$ and the data point index $i$ for the proof.  The inner loop of the adversary's updates can be written as
\begin{align}
\eta^{j,\ell+1} = \eta^{j,\ell} + \alpha\nabla_\eta f_0(x + \eta)^\top p_1^{\eta^{j,0}},
\end{align}
where $j$ is the iteration of the adversary's outer loop.  Recall that the true gradient of $\curlyA(\eta^{j,\ell})$ is,
\begin{align}
\nabla_\eta \curlyA(\eta^{j,\ell}) = \nabla_\eta f_0(x + \eta)^\top p_1^{\eta^{j,\ell}}.
\end{align}
We will bound the maximum difference of the update vector to the true gradient, over the iterations of the adversary's updates.  In this sense, the adversary's updates can be viewed as standard gradient method with an inexact gradient oracle.  We write,
\begin{align}
\|\nabla_\eta f_0(x + \eta)^\top p_1^{\eta^{j,0}} - \nabla_\eta \curlyA(\eta^{j,\ell})\| &= \|\nabla_\eta f_0(x + \eta)^\top p_1^{\eta^{j,0}} - \nabla_\eta f_0(x+\eta)^\top p_1^{\eta^{j,\ell}}\| \\
&\leq \|p_1^{\eta^{j,\ell}}-p_1^{\eta^{j,0}}\| \| \nabla_\eta f_0(x + \eta)\|.
\end{align}
Using the Lipschitz assumption on $f_0$, Lemma \ref{lem: del p bound}, and the definition of $C$ gives the bound
\begin{align} \label{eq: adv oracle bound}
\|\nabla_\eta f_0(x + \eta)^\top p_1^{\eta^{j,0}} - \nabla_\eta \curlyA(\eta^{j,\ell})\| \leq C \alpha (n-1).
\end{align} 
Equation \eqref{eq: adv oracle bound} shows that the updates are $C \alpha(n-1)$-approximate gradients for $\curlyA(\eta^{j,\ell})$.  

We now appeal to an inexact oracle convergence result in \cite{devolder2013first}.  Given a concave function $f(y)$ and a point $y$, we define a $(\delta, \mu, L)$ oracle as returning a vector $g(y)$ such that the following inequality holds.
\begin{align} \label{oracle ineq}
    \frac{\mu}{2} \|x - y\|^2 \leq f(x) - f(y) + \langle g(y), y - x\rangle \leq \frac{L}{2}\|x - y\|^2 + \delta.
\end{align}

It can be shown that if we have an approximate gradient bound of the form \eqref{eq: adv oracle bound}, and $\curlyA$ is $L_{\eta \eta}$-smooth and $\mu$-strongly concave in $\eta$, then the updates for the adversary are created by a $(\delta, \mu/2, 2L_{\eta \eta})$-oracle, where 
\begin{align}
\delta = C^2 \alpha^2 (n-1)^2\left(\frac{2}{\mu} + \frac{1}{2L_{\eta \eta}}\right).
\end{align}
Letting $\alpha < 1/L_{\eta \eta}$ and applying Theorem 4 in \cite{devolder2013first}, along with the inequality $\|\nabla \curlyA(\hat \eta)\|^2 \leq 2L_{\eta \eta}(\max_\eta \curlyA( \eta) - \curlyA(\hat \eta))$ from the $L_{\eta \eta} $ smoothness of $\curlyA$ in $\eta$ gives
\begin{align}
\| \nabla_\eta \curlyA_i(\hat \eta_i,\theta)\|^2\leq  L_{\eta \eta}^2 \|\eta^{0,0} - \eta^\star\|^2 \left(1 - \frac{\mu}{L_{\eta \eta}}\right)^{mn+1} + \frac{2 C^2}{L_{\eta \eta}} (n-1)^2\left(\frac{2}{\mu} + \frac{1}{2L_{\eta \eta}}\right).
\end{align}
Where $\eta^\star$ is the true solution to the inner maximization problem.  Since we initialize $\eta^{0,0} \in \mathcal X$ we have that $\|\eta^{0,0} - \eta^\star\|^2 \leq D(\mathcal X)$, which completes the proof.

\end{proof}

We are now ready to prove Theorem \ref{our theorem}.

\begin{proof}
In this proof, we define $\hat g_\curlyB(\theta^k) := \frac{1}{B} \sum_{i \in \curlyB} \nabla_\theta \curlyA_i(\hat \eta_i^k,\theta^k)$, where $\hat \eta_i^k$ is the output of the adversary's inner problem at iteration $k$, and $g_\curlyB(\theta^k) := \frac{1}{B} \sum_{i \in \curlyB}\nabla_\theta \curlyA_i(\eta_i^\star(\theta^k),\theta^k) $ is the true stochastic gradient.  With this notation, the updates for $\theta^k$ are
\begin{align}
\theta^{k+1} = \theta^k - \gamma_t \hat g_\curlyB(\theta^k).
\end{align}
The proof proceeds similarly as in \cite{ghadimi2013stochastic}, \cite{sinha2018certifiable}, and \cite{wang2019convergence}.  We begin with the inequality for the $L$-smoothness of $\mathcal R(\theta)$ applied to two successive iterates $\theta^k$ and $\theta^{k+1}$.
\begin{align}
\curlyR(\theta^{k+1}) &\leq \curlyR(\theta^k) + \langle \nabla L(\theta^k), \theta^{k+1} - \theta^k \rangle + \frac{L}{2} \| \theta^{k+1} - \theta^k\|^2 \\
&= \curlyR(\theta^k) + \langle \nabla \curlyR(\theta^k), - \gamma_t \hat g_\curlyB(\theta^k) \rangle + \frac{L \gamma_t^2}{2} \| \hat g_\curlyB(\theta^k)\|^2 \\
&= \curlyR(\theta^k) - \gamma_t\left(1 - \frac{L \gamma_t}{2}\right) \| \nabla \curlyR(\theta^k)\|^2 \nonumber\\
&\quad + \gamma_t\left(1 - \frac{L \gamma_t}{2}\right) \langle \nabla \curlyR(\theta^k), \nabla \curlyR(\theta^k) - \hat g_\curlyB(\theta^k)\rangle + \frac{L\gamma_t}{2} \| \hat g_\curlyB(\theta^k) - \nabla \curlyR(\theta^k)\|^2 \\
&\leq \curlyR(\theta^k) - \frac{\gamma_t}{2}\left(1 - \frac{L \gamma_t}{2}\right)\|\nabla \curlyR(\theta^k)\|^2 + \frac{\gamma_t}{2}\left(1 - \frac{L \gamma_t}{2}\right) \|\hat g_\curlyB(\theta^k) - g_\curlyB(\theta^k) \|^2\nonumber\\
& \quad + \gamma_t \left(1 + \frac{L \gamma_t}{2}\right) \langle \nabla \curlyR(\theta^k), \nabla \curlyR(\theta^k) - g_\curlyB(\theta^k)\rangle \nonumber\\
& \quad + L \gamma_t^2 \left(\|\hat g_\curlyB(\theta^k) - g_\curlyB(\theta^k) \|^2 + \|g_\curlyB(\theta^k) - \nabla \curlyR(\theta^k)\|^2\right).
\end{align}
Rearranging, using Theorem \ref{thm nesty} and Lemma \ref{lem: approx gradients}, taking the expectation with respect to the randomness of $\curlyB$ and conditioning on $\theta^k$ gives
\begin{align}
\E[\curlyR(\theta^{k+1}) - \curlyR(\theta^k)\;|\;\theta^k] &\leq -\frac{\gamma_t}{2}\left(1 - \frac{L \gamma_t}{2}\right)\| \nabla \curlyR(\theta^k)\|^2\nonumber \\
& \quad + \frac{\gamma_t}{2}\left(1 - \frac{L \gamma_t}{2}\right) \frac{L_{\theta \eta}}{\mu}\left(D(\mathcal X)L_{\eta \eta}^2  \left(1 - \frac{\mu}{L_{\eta \eta}}\right)^{mn+1}\!\!\!\! + \frac{2 C^2}{L_{\eta \eta}} (n-1)^2\left(\frac{2}{\mu} + \frac{1}{2L_{\eta \eta}}\right)\right) \nonumber\\
& \quad + L \gamma_t^2\left( \frac{L_{\theta \eta}}{\mu}\left(D(\mathcal X)L_{\eta \eta}^2  \left(1 - \frac{\mu}{L_{\eta \eta}}\right)^{mn+1}\!\!\!\! + \frac{2 C^2}{L_{\eta \eta}} (n-1)^2\left(\frac{2}{\mu} + \frac{1}{2L_{\eta \eta}}\right)\right) + \sigma^2\right).
\end{align}
Taking the expectation of both sides again and summing these inequalities from $k=0$ to $N-1$ gives
\begin{align}
\sum_{k=0}^{N-1} \frac{\gamma_t}{2}\left(1 - \frac{L \gamma_t}{2}\right) \E[\|\nabla \curlyR(\theta^k)\|^2] &\leq \E[\curlyR(\theta^0) - \curlyR(\theta^K)] + L \sum_{k=0}^{N-1} \gamma_t^2 \sigma^2 \nonumber \\
& \quad + \sum_{k=0}^{N-1} \frac{\gamma_t}{2}\left(1 + \frac{3L \gamma_t}{2}\right)\frac{L_{\theta \eta}}{\mu}\Bigg(D(\mathcal X) L_{\eta \eta}^2 \left(1 - \frac{\mu}{L_{\eta \eta}}\right)^{mn+1} \nonumber \\
& \qquad \qquad+ \frac{2 C^2}{L_{\eta \eta}}2 (n-1)^2\left(\frac{2}{\mu} + \frac{1}{2L_{\eta \eta}}\right)\Bigg).
\end{align}
The proof is completed by noting that $\curlyR(\theta^K) \geq \curlyR(\theta^\star)$, dividing through by $N$, and setting $\gamma_t = \gamma = \text{min}\{1/L, \sqrt{(\curlyR(\theta^0) - \curlyR(\theta^\star))/(L \sigma^2 N)}\}$.
\end{proof}

\bibliography{bibliography}

\end{document}